\documentclass[12pt]{amsart}
\usepackage[utf8]{inputenc}
\usepackage{amsmath}
\usepackage{amssymb}
\usepackage{amsfonts}
\usepackage{amsthm}
\usepackage{mathrsfs} 
\usepackage{bm}
\usepackage{bbm}
\usepackage{tikz}
\usepackage{centernot}
\usepackage{hyperref}
\usepackage[shortlabels]{enumitem}

\usepackage[ruled,vlined]{algorithm2e}
\usepackage[margin=1.2in]{geometry}

\DeclareMathOperator{\Id}{Id}

\DeclareMathOperator{\ord}{ord}

\newcommand{\Z}{\Bbb{Z}}

\renewcommand{\P}{\Bbb{P}}

\newcommand{\CC}{\mathcal{C}}

\newcommand{\ep}{\epsilon}

\hypersetup{
    colorlinks=true,
    linkcolor=blue,
    filecolor=magenta,
    urlcolor=blue,
    citecolor=blue
}

\makeatletter
\newtheorem*{rep@theorem}{\rep@title}
\newcommand{\newreptheorem}[2]{%
\newenvironment{rep#1}[1]{%
 \def\rep@title{#2 \ref{##1}}%
 \begin{rep@theorem}}%
 {\end{rep@theorem}}}
\makeatother

\makeatletter
\renewenvironment{proof}[1][\proofname]{\par
  \vspace{-\topsep}% remove the space after the theorem
  \pushQED{\qed}%
  \normalfont
  \topsep0pt \partopsep0pt % no space before
  \trivlist
  \item[\hskip\labelsep
        \itshape
    #1\@addpunct{.}]\ignorespaces
}{%
  \popQED\endtrivlist\@endpefalse
  \addvspace{6pt plus 6pt} % some space after
}
\makeatother

\newtheorem{thm}{Theorem}
\newreptheorem{thm}{Theorem}

\newtheorem{result}{Result}[section]
\newtheorem{lem}[result]{Lemma}
\newtheorem{prp}[result]{Proposition}
\newtheorem{cor}[result]{Corollary}

\newtheorem{clm}[result]{Claim}

\theoremstyle{definition}
\newtheorem{rmk}[result]{Remark}

\newtheorem*{ack}{Acknowledgements}

\theoremstyle{remark}

\newcommand{\hide}[1]{}
\newcommand{\edit}[1]{}%{\color{red}{#1}}}

\newcommand{\rough}[1]{}%\textbf{\textcolor{blue}{#1}}}
\definecolor{darkgreen}{RGB}{75,150,75}
\newcommand{\review}[1]{}%\textcolor{darkgreen}{#1}}

\newcommand{\hides}[1]{}%1}
\newcommand{\pub}[1]{}%\textcolor{purple}{#1}}

\title{Lower bounds for multicolor van der Waerden numbers}
\author{Zach Hunter}%\footnote{Institute of Science and Technology Austria, 3400 Klosterneuburg, Austria, \texttt{zachary.hunter@exeter.ox.ac.uk}}}
\address{Mathematical Institute, University of Oxford}
\email{zachary.hunter@exeter.ox.ac.uk}
\date{\today}

\begin{document}

\maketitle
\begin{abstract}We give an exponential improvement to the diagonal van der Waerden numbers for $r\ge 5$ colors.\end{abstract}

\section{Introduction}\label{intro}

For positive integer $n$, we write $[n]$ to denote $\{1,\dots,n\}$.

Given integers $k,r$, the van der Waerden number $w(k;r)$ is the smallest integer $N$ such that for any $r$-coloring $c:[N]\to [r]$, there exists a monochromatic $k$-term arithmetic progression. Currently, the best known upper bound
\[w(k;r) <2^{2^{r^{2^{2^{k+9}}}}}\]comes from Gowers' work on Szemer\'edi's theorem \cite{gowers}. Meanwhile, a lower bound of \[w(k;r)> \frac{r^{k-1}}{4k}\]follows from a result of Erd\H{o}s and Lov\'asz on the chromatic number of $k$-uniform hypergraphs in terms of maximum degree \cite{erdos}. Since then, there have been slight improvements to the lower bound, but only by factors that grow polynomially in $k$ (see e.g., \cite{berlekamp,kozik,szabo}).

Alternatively, one could discuss the inverse function $f_r(N)$, which is the smallest $k$ such that there exists an $r$-coloring $c:[N]\to [r]$ avoiding monochromatic arithmetic progressions of length $k$. The aforementioned bounds now state that
\[\log_{(5)} N - O_r(1)\le f_r(N) \le O\left(\frac{\log N}{\log r}\right) +O_r(1)\]
(here $\log_{(T)}$ denotes the $T$-times iterated logarithm (in base 2)).

In this paper, we shall improve the lower bound of $w(k;r)$ to the following.
\begin{thm}\label{main} For $r \ge 2$ with $r= a+3b$ (where $a\in \{2,3,4\}$), we have
\[w(k;r) > (a3^b)^{(1-o_r(1))k}.\]Alternatively, in terms of the inverse function, we prove \[f_r(N) \le O\left(\frac{\log N}{r}\right) +O_r(1).\]
\end{thm}
\begin{rmk}Theorem~\ref{main} gives an improved lower bound for $r\ge 5$ (when $k$ is large with respect to $r$).
\end{rmk}
%\noindent \hide{In other words, letting $C_r:= \liminf_{k\to\infty} w(k;r)^{1/k}$, we show that $C_r$ grows at least exponentially fast in $r$, rather than polynomially.}

\hide{Using a more involved proof, we can also improve in the case of $r=4$.
\begin{thm}\label{epgain} There exists some $\ep>0$ so that 
\[w(k;4)> (4+\ep)^k\] for all sufficiently large $k$.
\end{thm}\noindent We note that the arguments of Theorems~\ref{main} and \ref{epgain} can be combined to prove that $w(k;r) \ge (a(3+\ep')^b)^k$ for some absolute constant $\ep'>0$.}

Theorem~\ref{main} is achieved via a ``blow-up construction''. We note that in the analogous graph setting of Ramsey numbers, this type of blow-up argument is fairly trivial to pull off (indeed, this was done in a two-page paper by Lefmann in the 80's \cite{lefmann}). But in the arithmetic setting, the execution is less obvious and requires us to introduce randomness (in contrast to the deterministic blow-up methods available for graphs). We manage to achieve this by using a surprisingly useful trick involving direct products, which we believe is a novel technique for this area. \edit{I should probably check if similar ideas appear in the work of Berlekamp, though I find the paper hard to read}

\begin{rmk}We note that one can easily modify our arguments to slightly generalize our intermediate results and streamline some of our proofs. We omit such modifications here to avoid introducing unnecessary group-theoretic notation (namely, short exact sequences). The interested reader may find a write-up of our more general argument in \cite{huntergen}. 

\end{rmk}

\begin{ack}We thank Ben Green for useful conversations about related work (namely \cite{hunter}), where he pointed out that an argument could be more naturally described by thinking about direct products and short exact sequences. By following this suggestion, we gained a much better understanding of our argument, and eventually realized the results of this paper could be obtained. We also thank Ben Green for his feedback on an earlier version of this manuscript.

We additionally thank Daniel Altman, Zachary Chase, and Benny Sudakov for taking a look at this manuscript and giving helpful comments.

The vast majority of this work was done while the author was doing an internship at IST Austria. We are very thankful for their hospitality.
\end{ack}

\section{Preliminaries}\label{prelim}

In this paper, we use standard asymptotic notation. Specifically, given two functions $f=f(n),g=g(n)$, we say $f= O(g)$ or $f\gg g$ if there exists a constant $C>0$ such that $f(n) \le Cg(n)$ for all sufficiently large $n$. Also, we say $f = o(g)$ if $f(n)/g(n) \to 0$ as $n\to \infty$.  

%Also, given two subsets $A,B\subset G$ of a group, we write $A+B$ to denote the \textit{sumset} $\{a+b:a\in A,b\in B\}$.

For the purposes of this paper, it suffices to restrict our discussion to abelian groups. Hence, we will use additive notation. \hide{However, we note that nearly all of our work can be extended to the non-abelian setting in a meaningful way, see Remark~\ref{nonab} for more discussion.}

\hide{In this paper, we will need to generalize the concept of arithmetic progressions (AP's) of integers to the setting of arbitrary groups. To streamline notation, we adopt the convention that elements of a group ``remember'' the group they belong to, and will freque}

Given a group $G$ and integer $k$, a \textit{$k$-AP} is a set of the form $P = \{x+id:i\in\{0,\dots,k-1\}\}$ for some $x\in G,d\in G$; we say that $P$ is \textit{non-trivial} if $|P|>1$. We say a subset $S\subset G$ is \textit{$k$-AP-free} if it does not contain any non-trivial $k$-AP's. \edit{perhaps add some note about abuse of notation, since definitions implicitly depend on the group these elements belong to (similar to how Tao and Vu typically suppress the ambient group of additive sets)}

Also, for $d\in G$ and a $k$-AP $P\subset G$, we say $P$ has \textit{common difference $d$} if there exists $x \in G$ so that $P = \{x,x+d,\dots,x+(k-1)d\}$. We note that some $k$-AP's might not have a unique common difference, but they always have at least one. Additionally, we shall use the fact that a $k$-AP is non-trivial if and only if it has a common difference $d\neq 0_G$.

Now, given a group $G$ and integer $r$, we define $\kappa(G;r)$ to be the smallest integer $k$ such that there exists a coloring $c:G\to [r]$ that does not have monochromatic (non-trivial) $k$-AP's (i.e., each color class of $c$ is $k$-AP-free).

Finally, we will make use of direct products of groups. In what follows, we will mostly be exploiting the following fact: if $G = H_1\times H_2$, then the homomorphisms \[\pi_1:G \to H_1;(a,b)\mapsto a,\]\[\pi_2:G\to H_2;(a,b)\mapsto b\]are such that for each $g\in G\setminus \{0_G\}$, either $\pi_1(g)\neq 0_{H_1}$ or $\pi_2(g)\neq 0_{H_2}$ (or in other words, $\ker(\pi_1)\cap \ker(\pi_2) = \{0_G\}$).

\subsection{Basic facts}

Later on, we shall require the following well-known fact.
\begin{lem}\label{modN}Let $N_1,N_2\ge 1$ be coprime, and set $N= N_1N_2$. Then $\Z/N\Z \cong \Z/N_1\Z \times \Z/N_2\Z$.
\begin{proof}Write $G = \Z/N\Z,H_1= \Z/N_1\Z,H_2 = \Z/N_2\Z$. Obviously, $|G|= |H_1||H_2| = N$, thus it suffices to confirm that $G$ is cyclic (i.e., that there is some $g\in G$ such that $\ord(g):= \inf\{k>0:kg = 0_G\}$ is equal to $|G|$).

We consider $g= (1+N_1\Z,1+N_2\Z)\in G$. It is clear that $kg = 0_G$ if and only if $N_1\mid k$ and $N_2\mid k$. Since $N_1,N_2$ are coprime, we quickly see that $\ord(g) = N_1N_2 = N = |G|$, so $G$ is cyclic as desired.
\end{proof}
\end{lem}

We will also often implicitly make use of the following fact.
\begin{prp}Let $\pi:G\to H$ be a homomorphism. If $P\subset G$ is a $k$-AP with common difference $d$, then $\pi(P)\subset H$ is a $k$-AP with common difference $\pi(d)$.\begin{proof}Essentially immediate from definitions ($P = \{g,g+d,\dots,g+(k-1)d\}$ for some $g\in G$, thus $\pi(P) = \{\pi(g),\pi(g)+\pi(d),\dots,\pi(g)+(k-1)\pi(d)\}$).\end{proof}\end{prp}

\hide{\begin{rmk}\label{nonab}For those interested in the non-abelian setting, we can similarly define a $k$-AP in a non-abelian group $G$ to be a set $P =\{gh^i:i\in \{0,\dots,k-1\}\}$ for for some $g\in G,h\in G\setminus \{\Id_G\}$. One can then similarly extend the notions of $k$-AP-free sets, common differences, and the function $\kappa(G;r)$ to the non-abelian setting. We remark that our arguments almost never implicitly assume that unspecified groups are abelian (the only exception occurs in the proof of Lemma~\ref{cosetfree}, where we gloss over the difference between left-cosets and right-cosets, but this is easily resolved). Hence our proofs should continue to work and give the same results in this more general setting.
\end{rmk}}

\section{General Machinery}\label{mach}

We first need a key lemma, which we break into two parts.

\begin{lem}\label{cosetfree}Let $G = H_1\times H_2$ and for $i=1,2$ define the homomorphism $\pi_i:G\to H_i;(h_1,h_2)\mapsto  h_i$.

For each $x\in H_1$, choose some $k$-AP-free subset $Y_x \subset H_2$.

Set
\[A:= \{(x,y):x\in H_1,y\in Y_x\}.\]Let $P\subset G$ be a $k$-AP with common difference $d = (0_{H_1},d')$ for some $d'\neq 0_{H_2}$, then $P$ is not contained in $A$.
\begin{proof}Consider any $g=(x,y)\in G$. %Let $d' = \pi_2^{-1}(d)\in H_2\setminus\{0_{H_2}\}$, and consider any $g \in G$. Set $x= \pi_1(g),y=\pi_2(g)$.

We observe that
\[\pi_2(\{g,g+d,\dots, g+(k-1)d\}\cap A)= \{y,y+d',\dots, y+(k-1)d'\}\cap Y_x.\]Since $Y_x\subset H_2$ is $k$-AP-free, and $d'\neq 0_{H_2}$, we have that $\{y,y+d',\dots,y+(k-1)d'\}\not\subset Y_x$. So it then quickly follows that $\{g,g+d,\dots ,g+(k-1)d\}\not\subset A$, as desired. The result follows.\end{proof}
\end{lem}

\begin{lem}\label{shortex}Let $G = H_1\times H_2$ and for $i=1,2$ define the homomorphism $\pi_i:G\to H_i;(h_1,h_2)\mapsto  h_i$. 

Suppose we have sets $S=\{x_1,\dots,x_m\}\subset H_1$ and $Y_1,\dots,Y_m\subset H_2$ that are each $k$-AP-free in their respective groups.

Then,
\[A := \bigcup_{i=1}^m\{(x_i,y): y\in Y_i\}\]is $k$-AP-free with respect to $G$.
\begin{proof}Consider any $g\in G$ and $d \in G\setminus \{0_G\}$. 

Suppose for sake of contradiction that $P:= \{g,g+d,\dots,g+(k-1)d\}\subset A$. Then, we must clearly have
\[\pi_1(\{g,g+d,\dots,g+(k-1)d\}) \subset \pi_1(A) = S.\]Because $S\subset H_1$ is $k$-AP-free, and $\pi_1(P)\subset H_1$ is a $k$-AP with common difference $\pi_1(d)$, this means that $\pi_1(d) = 0_{H_1}$ must hold.

Now by the assumption $d\neq 0_G$, it follows that $d =(0_{H_1},d')$ for some $d'\neq 0_{H_2}$. We are then done by appealing to Lemma~\ref{cosetfree}. Indeed, as the empty set is $k$-AP-free, we see that our set $A$ satisfies the conditions of Lemma~\ref{cosetfree}. Thus it is impossible for $A$ to contain a non-trivial $k$-AP with common difference $d$ (which we just assumed is of the form $(0_{H_1},d')$ for $d'\neq 0_{H_2}$), giving us our contradiction. \end{proof}
\end{lem}

We can now present a sufficient condition for when we can do a ``blow-up construction''. In Section~\ref{trick}, we will proceed to obtain a more convenient consequence of the below (Lemma~\ref{synth}), which will allow us to deduce Theorem~\ref{main}. 
\begin{thm}\label{blackbox}
Let $r_1,r_2,r_3,k$ be positive integers, and $\delta>0$ be some constant.

Let $G = H_1\times H_2$ and for $i=1,2$ define the homomorphism $\pi_i:G\to H_i;(h_1,h_2)\mapsto h_i$.

Also, suppose that $\ord(H_1) \ge Q$.

Furthermore, suppose there exist colorings $C_1:H_1\to [r_1],C_2:H_2\to [r_2+r_3]$ such that:
\begin{enumerate}[label={(\arabic*)}]
    \item\label{1a} the color classes of $C_1,C_2$ are both $k$-AP-free;
    \item\label{2a} $|C_2^{-1}(r_2+[r_3])|\le \delta |H_2|$;
    \item\label{3a} and $|G|^2\le \delta^{-\min\{Q,k\}}$.
\end{enumerate}
Then, there exists a coloring $c:G\to [r_1r_2+r_3]$ that avoids monochromatic non-trivial $k$-AP's.
\begin{proof}We shall construct a coloring $\CC:G\to ([r_1]\times [r_2]) \cup [r_3]$ randomly, and prove that $\CC$ avoids monochromatic non-trivial $k$-AP's with positive probability. By fixing an outcome without monochromatic non-trivial $k$-AP's and identifying $([r_1]\times [r_2])\cup [r_3]$ with $[r_1r_2+r_3]$, we get our desired $c$.

For each $x \in H_1$, we define $y_x$ to be a element of $H_2$ chosen uniformly at random (and independently of all other random variables). Then, for $g= (x,y) \in G$, we set
\[\CC(g) =\begin{cases}(C_1(x),C_2(y-y_x))&\textrm{ if }C_2(y-y_x) \in [r_2]\\
C_2(y-y_x)-r_2&\textrm{ otherwise.}
\end{cases}\]
It is straight-forward to verify that $\CC$ is well-defined on $G$, and takes values in $([r_1]\times [r_2])\cup  [r_3]$.

We are left to prove that $\CC$ lacks monochromatic non-trivial $k$-AP's with positive probability. For $(i,j)\in [r_1]\times [r_2]$, we always have that $\CC^{-1}((i,j))$ is $k$-AP-free by Lemma~\ref{shortex}. Also, by Lemma~\ref{cosetfree}, for every $d = (0_{H_1},d')$ with $d'\neq 0_{H_2}$, we have that any $k$-AP $P\subset G$ with common difference $d$ is not monochromatic under $\CC$.

It remains to consider $k$-AP's $P$ with common difference $d \in G \setminus \pi_1^{-1}(0_{H_1})$. We shall proceed by a union bound. By counting the ways to choose $g,d$, we see there are at most $|G|(|G|-1)<|G|^2$ such $k$-AP's $P\subset G$. Also, by the above, we only need to worry about the color classes $\CC^{-1}(i)$ for $i\in [r_3]$.

Hence it suffices to show that for each $k$-AP $P\subset G$ with common difference $d\in G\setminus \pi_1^{-1}(0_{H_1})$,
\[\P(P\subset \CC^{-1}([r_3]))\le 1/|G|^2.\]

We fix an arbitrary such $P$. By assumption, $P$ has a common difference $d \in G\setminus \pi_1^{-1}(0_{H_1})$. It then follows that $\pi_1(d) \neq 0_{H_1}$. Hence, by our assumption that $\ord(H_1)\ge Q$, we have that $\pi_1(P)$ takes at least $\ell:= \min\{Q,k\}$ distinct values $x_1,\dots,x_\ell \in H_1$. Thus, by the independence of the variables $y_{x_1},\dots,y_{x_\ell}$, we have
\[\P(P \subset \CC^{-1}([r_3])) \le \left(\frac{|C_2^{-1}(r_2+[r_3])|}{|H_2|}\right)^\ell.\]By Properties~\ref{2a} and \ref{3a}\hide{our assumption on $|C_2^{-1}(r_2+[r_3])|$}, the LHS is at most $\delta^\ell \le 1/|G|^2$ as desired.
\end{proof}
\end{thm}

\section{A sparsification trick}\label{trick}

We first need the following construction of $k$-AP-free sets, which was originally observed by Erd\H{o}s and Tur\'an in \cite{erdosturan}. We provide a short proof of the statement which mimics the ideas from our proof of Lemma~\ref{shortex}.

\begin{prp}\label{sparsify}Consider a prime $p$ and some integer $t\ge 1$. Let $N= p^t$ and $G = \Z/N\Z$.

There exists a $p$-AP-free set $S\subset G$ with $|S| = (p-1)^t = (1-1/p)^t N$. 
\begin{proof}Start by defining $A_1 = \{1,\dots,p-1\} = [p-1]\subset \Z$. Then, for $t\ge 1$, let $A_{t+1} =  A_1+p\cdot A_t = \{a_1+pa_t:a_1\in A_1,a_t\in A_t\}$. In other words, $A_t$ shall be the set of integers $n\in [p^t]$ using only digits from $A_1$ in base $p$.

We claim that we may take $S_t = A_t+p^t\Z \subset \Z/p^t\Z$. It is clear that $|S_t| = |A_t| = (p-1)^t$, as desired, so it remains to check that $S_t$ is $p$-AP-free. We shall induct on $t$.

First, when $t=1$, we note that the only non-trivial $p$-AP $P\subset \Z/p\Z$ is $\Z/p\Z$ itself (since each $d\neq 0_{\Z/p\Z}$ generates $\Z/p\Z$). Thus as $S_1$ is a proper subset of the group, it will not contain such $P$. Thus $S_1$ is $p$-AP-free.

Now assuming $S_1$ and $S_t$ is $p$-AP-free for some $t\ge 1$, we'll show the same holds for $S_{t+1}$. Take any $p$-AP $P\subset \Z/p^{t+1}\Z$ and suppose $P\subset S_{t+1}$. Considering the projection \[\pi:\Z/p^{t+1}\Z\to \Z/p\Z; n+p^{t+1}\Z\mapsto n+p\Z,\] we get that \[\pi(P)\subset \pi(S_{t+1}) = S_1.\] Thus as $S_1$ is $p$-AP-free, $\pi(P)\subset \Z/p\Z$ must be a trivial $p$-AP. In particular this means that $P$ has common difference $d = pd'+p^{t+1}\Z$ for some integer $d'$.

Next, in the spirit of Lemma~\ref{cosetfree}, we notice that $P$ now corresponds to a $p$-AP in $\Z/p^t\Z$ with common difference $d=d'+p^t\Z$. Specifically, writing $P = \{g,g+d,\dots,g+(p-1)d\}$ for some $g = g_0+pg'+p^{t+1}\Z$ and $P' = \{g',g'+d',\dots,g'+(p-1)d'\}+p^t\Z\subset \Z/p^t\Z$, we have that
\[(P\cap S_{t+1})-g_0 = p\cdot (P'\cap S_t)\](assuming $p\nmid g_0$, because otherwise the RHS would be empty, contradicting the assumption that $P\subset S_{t+1}$). Since we are assuming $P\subset S_{t+1}$, this should imply that $P' \subset S_t$. As $S_t$ is $p$-AP-free, $P'$ must be trivial, meaning that $p^t|d'$ and thus $d = 0_{\Z/p^{t+1}\Z}$ (making $P$ trivial as well). Consequently, $S_{t+1}$ is $p$-AP-free, as it does not contain non-trivial $p$-AP's.\end{proof}
\end{prp}
\begin{rmk}Secretly, what we've done is applied a generalized version of Lemma~\ref{shortex}, using the fact that $\Z/p^{t+1}\Z = \Z/p^t\Z \rtimes \Z/p\Z$ (i.e., replacing direct products with semi-direct products). Further details on such ideas are given in \cite{huntergen}.
\end{rmk}

We can now prove the following technical lemma, which is the synthesis of everything proven thus far. 

We remind our readers that for a group $G$ and integer $r$, that $\kappa(G;r)$ denotes the minimum $k$ such that there exists an $r$-coloring of $G$ where each color class is $k$-AP-free. \edit{consider inserting more naturally}
\begin{lem}\label{synth}Consider positive integers $r,r',k,Q$.

Let $H_1$ be any group and take $H_2 = \Z/p^t\Z$ for some prime $p\le k$. Now let $G= H_1\times H_2$.

Furthermore suppose that:
\begin{enumerate}[label={(\arabic*)}]
    \item\label{1b} we have\[\max\{\kappa(H_1;r),\kappa(H_2;r')\}\le k;\]
    \item\label{2b} we have $\ord(H_1) \ge Q$;
    \item\label{3b} we have $(1-(1-1/p)^t)^{-\min\{Q,k\}}\ge |G|^2$.
\end{enumerate}Then $\kappa(G;r+r')\le k$.
\begin{proof}Since we assume $\kappa(H_2;r')\le k$, there exists a coloring $c_2:H_2\to [r']$ avoiding monochromatic non-trivial $k$-AP's.

Let $\delta = 1-(1-1/p)^t$. Applying Proposition~\ref{sparsify}, we may find a $p$-AP-free (and hence $k$-AP-free, as $p\le k$) set $S\subset H_2$ such that $|H_2 \setminus S|\le \delta |H_2|$. We then define the coloring $C_2:H_2\to [r'+1]$, so that \[C_2(y) = \begin{cases}1 &\textrm{if }y\in S\\
1+c_2(y)&\textrm{otherwise.}
\end{cases}\]It is clear that $C_2$ is well-defined, and takes values in $[r'+1]$. Meanwhile, we see that $C_2$ lacks monochromatic non-trivial $k$-AP's, as each of its color classes is a subset of $k$-AP-free set. Finally, $C_2$ has the important property that $|C_2^{-1}(1+[r'])|= |H_2\setminus S|\le \delta|H_2|$.

Unpacking the rest of our assumptions, we may invoke Theorem~\ref{blackbox} with $r_1 = r,r_2 =1,r_3 = r'$ to get the desired result. 
\end{proof}\end{lem}

\section{Proof of Theorem 1}

We start by recalling the following coloring result of Erd\H{o}s and Lov\'asz.
\begin{prp}[{\cite[Theorem~2]{erdos}}]\label{hypercol} If $H$ is a $k$-uniform hypergraph with $\Delta(H)\le r^{k-1}/4k$ (i.e., each vertex is contained by at most $r^{k-1}/4k$ hyperedges), then $H$ has a proper $r$-coloring of $V(H)$.
\end{prp}
\begin{rmk}As noted in Section~\ref{intro}, Proposition~\ref{hypercol} gave (up to factors of $k^{O(1)}$) the previous best known lower bound for $w(k;r)$. Here, we will apply Proposition~\ref{hypercol} to groups, which loses an extra factor of $k$, but we will not be concerned about subexponential factors of shape $\exp(-o(k))$.
\end{rmk}

We now get the following corollary.

\begin{cor}\label{Gcol}Let $G$ be a group and $r,k$ be integers, where $\ord(G)\ge k$, and $|G| \le r^{k-1}/4k^2$.

Then $\kappa(G;r) \le k$.\begin{proof}
Consider the hypergraph $H$ with vertex set $V(H)=G$ and hyperedge set $E(H) = \{P:P\textrm{ is a }k\textrm{-AP}\}$. Since no $d\in G\setminus\{0_G\}$ has $\ord(d) < k$, we may conclude that $H$ is $k$-uniform (i.e., that every hyperedge has cardinality $k$). 

Next, we note that each vertex $v$ is contained in at most $k(|G|-1)$ hyperedges. Indeed, there are $|G|-1$ choices of the common difference $d \in G\setminus\{0_G\}$ and at most $k$ distinct $k$-AP's with common difference $d$ that can contain $v$. Hence, we have $\Delta(H)\le k(|G|-1)\le k|G|$ (where $\Delta(H)$ denotes the maximum degree of vertices in $H$). 

We can then invoke Proposition~\ref{hypercol}. Indeed, since $\Delta(H)\le k|G| \le \frac{r^{k-1}}{4k}$, the assumptions of Proposition~\ref{hypercol} are satisfied, and so there is a proper $r$-coloring of $V(H)$. In other words, there exists $c:V(H) \to [r]$ so that no edge of $H$ is monochromatic under $c$. By the definition of $H$, this means $c$ is a coloring of $G$ where every color class is $k$-AP-free, implying $\kappa(G;r)\le k$ as desired.
\end{proof}
\end{cor}

We are nearly able to deduce our main result. We just need the following convenient lemma.

\begin{lem}\label{workhorse} Fix $\ep \in (0,1/10),C>0$ and some integer $r\ge 2$. There exists an absolute constant $K =K(\ep,C,r)$ such that for all $k> K$ the following holds:

Suppose $p\in ((1-\ep)k,k]$ is prime and that $H_1$ is some group with 
$|H_1| \le C^k, \ord(H_1) \ge (1-\ep)k$ and $\kappa(H_1;r') \le k$ (for some $r'$). 

Then taking $t= \lfloor k(1-2\ep)\frac{\log r}{\log k}\rfloor$, and defining $H_2:= \Z/p^t\Z$, we have:
\begin{enumerate}[label={(\arabic*)}]
    \item \label{1c} $(1-\ep)r^{1-2\ep} \le |H_2|^{1/k} \le r^{1-2\ep}$;
    \item \label{2c} $\kappa(H_2;r)\le k$;
    \item \label{3c} $\kappa(H_1\times H_2;r+r') \le k$.
\end{enumerate}
\begin{proof}What follows is just some menial asymptotic calculations which tell us that Corollary~\ref{Gcol} and Lemma~\ref{synth} can both be invoked, giving the desired result. We encourage the reader to not dwell on the details. The main point is that since $t= o(k)$, we have that the $\delta$ from the proof of Lemma~\ref{synth} will be $o(1)$. And at the same time, we'll also have $|H_1\times H_2|^{1/k} = O(1)$. Thus Lemma~\ref{synth} can be used. 

As stated above, let $t = \lfloor k(1-2\ep)\frac{\log r}{\log k}\rfloor$. Now by definition, $k^t = cr^{(1-2\ep)k}$ for some $c\in [1/k,1]$.

So, assuming $p\in ((1-\ep)k,k]$, we get the bounds 
\[p^t \le k^t \le r^{(1-2\ep)k}\]and
\[p^t \ge (1-\ep)^t k^t \ge \frac{(1-\ep)^t}{k}r^{(1-2\ep)k}.\]
Furthermore, assuming $k$ is sufficiently large, the lower bound can be weakened to $p^t \ge (1-\ep)^k r^{(1-2\ep)k}$.

Now define $H_2 = \Z/p^t\Z$. Now obviously $|H_2| = p^t$ so by the above bounds, condition~\ref{1c} is satisfied.

Also, by our upper bound above, we have that $|H_2|\le r^{(1-\ep)k}/4k^2$ for sufficiently large $k$. Thus, by Lemma~\ref{Gcol} we have that $\kappa(H_2;r)\le (1-\ep)k$ (here we recall that $\ord(H_2) = p\ge (1-\ep)k$ to apply said lemma).

At last we consider some $H_1$ as in the statement, and seek to apply Lemma~\ref{synth}. Take $\delta = 1-(1-1/p)^t \le t/p \le \frac{(1-2\ep)\log r}{(1-\ep)\log k} = O_r(1/\log k)$ (here we use the fact that $(1-x)^t\ge 1-tx$ for $t\ge 1$).

So clearly, as $k\to \infty$, we have that $\delta \downarrow 0$. Thus, for sufficiently large $k$ (with respect to $\ep,C,r$), we have \[\delta^{-(1-\ep)k} \ge (Cr)^{2k} \ge |H_1\times H_2|^2\] (recalling $|H_1|\le C^k$ and $|H_2|\le r^{(1-2\ep)k} <r^k$). Hence condition~\ref{2c} is satisfied.

Thus, recalling the assumption $\ord(H_1) \ge (1-\ep)k$, we may invoke Lemma~\ref{synth} to get that $\kappa(H_1\times H_2;r+r') \le k$, as desired. Whence, we conclude condition~\ref{3c} is satisfied.\end{proof}
\end{lem}

We can now deduce Theorem~\ref{main}.
\begin{proof}[Proof of Theorem~\ref{main}]Fix $r=a+3b$ (where $a\in \{2,3,4\}$). For every $\ep \in (0,1/10)$, we shall show that whenever $k$ is sufficiently large with respect to $\ep,r$, we have $\kappa(\Z/N\Z;k)\le r$ for some $N > (1-\ep)^{(b+1)k}(a3^b)^{(1-2\ep)k}$. As $(1-\ep)^{b+1}(a3^b)^{(1-2\ep)} \to a3^b$ as $\ep \downarrow 0$, we see that $\kappa(\Z/N\Z;r)\le k$ for some some $N> (a3^b)^{(1-o(1))k}$, giving our desired result (since $\kappa(\Z/N\Z;r)\le k$ implies $w(k;r)>N$).

Now, fix some $\ep \in (0,1/10)$.

By the prime number theorem, there exists $K' = K_\ep'$ such that for $k>K'$ there exists $b+1$ distinct primes $p_0,\dots, p_b \in ((1-\ep)k,k]$.

Next let $K=\max\{K(\ep,1,a),K(\ep,a3^b,3)\}$ be the value given by Lemma~\ref{workhorse}, and take $K^* = \max\{K,K'\}$. 

Consider $k>K^*$. As $K^*\ge K'$, we may fix distinct primes $p_0,\dots,p_b\in ((1-\ep)k,k]$. 

Now take $t_0 = \lfloor k(1-2\ep)\frac{\log a}{\log k}\rfloor,t' = \lfloor k (1-2\ep)\frac{\log 3}{\log k}\rfloor $ (like in the statement of Lemma~\ref{workhorse}). We set $H_0 = \Z/p_0^{t_0}\Z$ and for $i=1,\dots, b$ we let $H_i = \Z/p_i^{t'}\Z$.

We define $G_{-1}$ to be the trivial group on one element, take $G_0 = G_{-1}\times H_0$, and for $i=1,\dots,b$ we define $G_i = G_{i-1}\times H_i$. By construction, we have that $|G_i|\le |G_b| \le (a3^b)^k$ for all $i=0,1,\dots,b$. We also note that $\ord(G_i) > (1-\ep)k$ for all $i=0,1,\dots,b$, since $\ord(G_i)$ must be some divisor of $|G_i|$ (besides one) by Lagrange's theorem.

Thus, by induction (applying Lemma~\ref{workhorse}), we'll have that $\kappa(G_b;r)\le k$. Finally, since $p_0,\dots, p_b$ are distinct primes, and $|H_i|$ is a power of $p_i$ for $i =0,1,\dots,b$, we may repeatedly apply Lemma~\ref{modN} to deduce that $G_b \cong \Z/N\Z$ where $N = \prod_{i=0}^b |H_i| \ge ((1-\ep)^{b+1}(a3^b)^{1-2\ep})^k$. So, we get our desired result.\end{proof}

\end{document}